\documentclass{imsart}
\usepackage{mathrsfs}

\usepackage[centertags,reqno]{amsmath} 
\usepackage[colorlinks=true, pdfstartview=FitV, linkcolor=blue,
  citecolor=blue, urlcolor=blue]{hyperref}
\usepackage{graphics}
\usepackage{color}
\usepackage{amsmath}
\usepackage{amsfonts}
\usepackage{amssymb}
\usepackage{amsthm}
\usepackage{amsbsy}
\usepackage{latexsym}
\usepackage{epsfig}
\usepackage{fullpage}
\usepackage{verbatim}
\usepackage[latin1]{inputenc}
\usepackage{mhequ}
\numberwithin{equation}{section}
\usepackage{subfigure}
\usepackage[section]{placeins}
\usepackage{enumitem}
\usepackage[numbers]{natbib}

\renewcommand{\nabla}{D}

\newcommand{\tr}{\mathrm{Tr}}

\newcommand{\bbR}{\mathbb{R}}

\definecolor{darkblue}{rgb}{0,0,0.7}

\newtheorem{theorem}{Theorem}[section]
\newtheorem{lemma}[theorem]{Lemma}
\newtheorem{remark}[theorem]{Remark}

\newtheorem{proposition}[theorem]{Proposition}
\newtheorem{properties}[theorem]{Properties}

\begin{document}

\begin{frontmatter}
\title{Analysis of the 3DVAR Filter for the\\
Partially Observed Lorenz '63 Model}

\begin{aug}


\author{K. J. H. Law\ead[label=e1]{k.j.h.law@warwick.ac.uk}},  
  \author{A. Shukla\ead[label=e2]{a.shukla@warwick.ac.uk}}
	\and
 \author{A. M. Stuart  \ead[label=e3]{a.m.stuart@warwick.ac.uk}}
  \address{Warwick Mathematics Institute\\
    University of Warwick\\ CV4 7AL, UK\\ 
           \printead{e1,e2,e3}}

\end{aug}

\maketitle
\begin{abstract}
The problem of effectively combining data with a 
mathematical model constitutes a major challenge
in applied mathematics. It is particular
challenging for high-dimensional dynamical systems
where data is received sequentially in time and
the objective is to estimate the system state
in an on-line fashion; this situation arises, for
example, in weather forecasting.
The sequential particle filter 
is then impractical and {\em ad hoc} filters, which
employ some form of Gaussian approximation, are
widely used. Prototypical of these {\em ad hoc} filters
is the 3DVAR method. The goal of this paper is
to analyze the 3DVAR
method, using the Lorenz '63 model to exemplify the
key ideas. The situation where the data is partial 
and noisy is studied, and both discrete time and
continuous time data streams are considered. The
theory demonstrates how the widely used technique
of variance inflation acts to stabilize the filter,
and hence leads to asymptotic accuracy. 

\end{abstract}
\end{frontmatter}

\maketitle

\section{Introduction}

Data assimilation concerns estimation of the state
of a dynamical system by combining observed data with the 
underlying mathematical model. It finds widespread
application in the geophysical sciences, including
meteorology \cite{kalnay2003atmospheric}, oceanography
\cite{ben02} and oil reservoir 
simulation \cite{oliver2008inverse}. 
Both filtering methods, which update the state sequentially,
and variational methods, which can use an entire
time window of data, are used \cite{apte2008data}. However,
the dimensions of the systems arising in the applications
of interest are enormous  -- of ${\cal O}(10^9)$ in global
weather forecasting, for example. This makes
rigorous Bayesian approaches such as the sequential particle 
filter \cite{doucet2001sequential}, for the filtering
problem, or MCMC methods for the variational problem 
\cite{article:Stuart2010}, prohibitively expensive in
on-line scenarios. 

For this reason various
{\em ad hoc} methodologies are typically used.
In the context of filtering these usually rely on
making some form of Gaussian ansatz \cite{VL09}. The 3DVAR
method \cite{article:Lorenc1986,article:Parrish1992}
is the simplest Gaussian filter, relying on
fixed (with respect to the data time-index increment)
forecast and analysis model covariances, related through
a Kalman update. A more sophisticated idea is to
update the forecast covariance via the linearized
dynamics, again computing the analysis covariance
via a Kalman update, leading to the extended Kalman filter
\cite{jazwinski1970stochastic}. In high dimensions 
computing the full linearized dynamics is not practical.
For this reason the ensemble Kalman filter 
\cite{evensen2003ensemble,evensen2009data} is widely
used, in which the forecast covariance is estimated from
an ensemble of particles, and each particle is updated
in Kalman fashion. An active current area of research
in filtering concerns the development of methods
which retain the computational expediency of approximate
Gaussian filters, but which incorporate physically
motivated structure into the forecast and analysis steps
\cite{majda2010mathematical,majda2012filtering}, and
are non-Gaussian.

Despite the widespread use of these
many variants on approximate Gaussian filters,
systematic mathematical analysis 
remains in its infancy. Because the 3DVAR method is
prototypical of other more sophisticated {\em ad hoc}
filters it is natural to develop a thorough understanding
of the mathematical properties of this filter.
Two recent papers address these issues in the
context of the Navier-Stokes equation, for
data streams which are discrete in time \cite{brett2012accuracy}
and continuous in time \cite{blomker2012accuracy}. 
These papers study the situation where
the observations are partial (only low frequency spatial
information is observed) and subject to small noise.
Conditions are established under which the filter can recover
from an order one initial error and, after enough time
has elapsed, estimate the entire system state
to within an accuracy level determined by the
observational noise scale; this is termed {\em filter accuracy.}
Key to understanding, and proving,
these results on the 3DVAR filter
for the Navier-Stokes equation are a pair of papers
by Titi and co-workers which study the synchronization
of the Navier-Stokes equation with a true signal
which is fed into only the low frequency spatial modes of the system, without
noise \cite{olson2003determining,hayden2011discrete}; the
higher modes then synchronize because of the underlying 
dynamics.  The idea that a finite amount of
information effectively governs the large-time
behaviour of the Navier-Stokes equation goes back to
early studies of the equation as a dynamical system
\cite{foias1967comportement} and is known as the 
{\em determining node} or {\em mode} property 
in the modern literature \cite{book:Robinson2001}. 
The papers \cite{brett2012accuracy, blomker2012accuracy}
demonstrate that the technique of {\em variance inflation}, 
widely employed by practitioners in high dimensional
filtering, can be understood
as a method to add greater weight to the data, thereby
allowing the synchronization effect to take hold.

The Lorenz '63 model \cite{lorenz1963deterministic,colin1982}
provides a useful metaphor for various aspects of the
Navier-Stokes equation, being dissipative with a
quadratic energy-conserving nonlinearity \cite{Titi2001}.
In particular, the Lorenz model exhibits a form of
synchronization analogous to that mentioned above for
the Navier-Stokes equation \cite{hayden2011discrete}.
This strongly suggests that results proved for 3DVAR
applied to the Navier-Stokes equation will have analogies
for the Lorenz equations. The purpose of this paper is
to substantiate this assertion.

The presentation is organized as follows.
In section \ref{sec:2} we describe the Bayesian formulation
of the inverse problem of sequential data assimilation; we 
also present a brief introduction to the relevant properties
of the Lorenz '63 model and describe the 3DVAR filtering schemes
for both discrete and continuous time data streams. 
In section \ref{sec:dt} we derive Theorem \ref{91} concerning
the 3DVAR algorithm applied to the Lorenz model with
discrete time data. This is analogous to Theorem 3.3
in \cite{brett2012accuracy} for the Navier-Stokes
equation. However, in contrast to that paper, we study
Gaussian (and hence unbounded) observational noise
and, as a consequence, our results are proved in mean
square rather than almost surely. In section \ref{sec:ct}
we extend the accuracy result to the continuous time data
stream setting: Theorem \ref{t:55}; the result
is analogous to Theorem 4.3 in \cite{blomker2012accuracy}
which concerns the Navier-Stokes equation.
Section \ref{sec:num} contains numerical results 
which illustrate the theory.
We make concluding remarks in section \ref{sec:conc}.

\section{Set-Up}
\label{sec:2}

In subsection \ref{ssec:ip} 
we formulate the probabilistic inverse
problem which arises from attempting to estimate the
state of a dynamical system subject to uncertain
initial condition, and given partial, noisy
observations. Subsection \ref{ssec:2.1} introduces
the Lorenz '63 model which we employ throughout
this paper. In subsections \ref{ssec:2.2} and
\ref{ssec:2.3} we describe the discrete and continuous 
3DVAR filters whose properties we study in
subsequent sections.

\subsection{Inverse Problem}
\label{ssec:ip}
Consider a model whose dynamics is governed by the equation
\begin{equation}
\frac{\textrm{d}u}{\textrm{d}t}=\mathcal{F}(u),\label{211}
\end{equation}
with initial condition $u(0)=u_0\in\mathbb{R}^p$. We assume the the initial condition is uncertain and only its statistical
distribution is known, namely the Gaussian $u_0\sim N(m_0,C_0)$.
Assuming that the equation has a solution for any
$u_0\in\mathbb{R}^p$ and all positive times, we
 let $\Psi(\cdot,\cdot):\mathbb{R}^p\times\mathbb{R}^+\rightarrow\mathbb{R}^p$ be the solution operator for equation \eqref{211}. 
Now suppose that we observe the
system at equally spaced times $t_k=kh$ for all 
$k\in \mathbb{Z}^+$. For simplicity we write $\Psi(\cdot):=
\Psi(\cdot;h).$ Defining $u_k=u(t_k)=\Psi(u_0;kh)$ we have 
\begin{equation}
\label{eq:map}
u_{k+1}=\Psi(u_k), \quad k\in \mathbb{Z}^+.
\end{equation}
We assume that the data $\{y_k\}_{k\in \mathbb{Z}^+}$ 
is found from noisily observing
a linear operator $H$ applied to the system state,
at each time $t_k$, so that 
\begin{equation}\label{212}
y_{k+1}=Hu_{k+1}+{\nu}_{k+1}, \quad k\in \mathbb{N}.
\end{equation}
Here $\{\nu_k\}_{k\in\mathbb{N}}$ is an i.i.d. sequence of random variables, independent of $u_0$, with $\nu_1\sim N(0,\Gamma)$ and $H$ denotes a linear operator from $\mathbb{R}^p$ to $\mathbb{R}^m$, with
$m \le p.$ If the rank of $H$ is less than $p$
the system is said to be {\em partially observed}. 
{The partially observed situation is the most commonly
arising in applications and we concentrate on it here. The 
over-determined case $m>p$ corresponds to making 
more than one observation in certain directions; one
approach that can be used in this situation 
is to average multiple observations to reduce the effective
observational error variance 
by the square root of the number of observations in that direction,
and thereby reduce to the case where the rank is less than or
equal to $p$.}

We denote the accumulated data up to time $k$ 
by $Y_k:={\{y_j\}}^k_{j=1}$.
The pair $(u_k,Y_k)$ is a jointly varying
random variable in $\mathbb{R}^p \times \mathbb{R}^{km}$.
The goal of filtering is to determine the distribution of 
the conditioned random variable $u_k|Y_k$, and to update it
sequentially as $k$ is incremented. {This corresponds to
a sequence of inverse problems
for the system state}, given observed data, and it 
has been regularized by means of the Bayesian formulation.

\subsection{Forward Model: Lorenz '63}
\label{ssec:2.1}

When analyzing the 3DVAR approach to the filtering
problem we will focus our attention on a particular
model problem, namely the classical Lorenz '63 system \cite{lorenz1963deterministic}.
In this section we introduce the model and summarize 
the properties relevant to this paper. 
The Lorenz equations are a system of three coupled non-linear ordinary differential equations whose solution $u\in\mathbb{R}^3$, where $u=(u_x, u_y, u_z)$, satisfies
\begin{subequations}
\begin{eqnarray}
\dot{u}_x&=&\alpha(u_y-u_x),
\label{32a}\\
\dot{u}_y&=&-\alpha u_x-u_y-{u_x}u_z, \label{32b}\\
\dot{u}_z&=&u_x{u_y}-b{u_z}-b(r+\alpha). \label{32c}
\end{eqnarray}
\end{subequations}
Note that we have employed a coordinate system where origin is shifted to the point $\bigl(0,0,-(r+\alpha)\bigr)$ as discussed in \cite{book:Temam1997}. Throughout this paper we will use the classical parameter values $(\alpha,b,r)=(10,\frac{8}{3},28)$ in 
all of our numerical experiments. At these values,
the system is chaotic \cite{WT2002} and
has one positive and one negative
Lyapunov exponent and the third is zero, reflecting time
translation-invariance.
Our theoretical results, however, simply require 
that $\alpha, b > 1$ and $r>0$ and we make this
assumption, without further comment, throughout the remainder
of the paper.

In the following it is helpful to write the Lorenz equation in the following form as given in \cite{Titi2001},\cite{hayden2011discrete}:
\begin{equation}
\frac{du}{dt}+Au+B(u,u)=f,\quad u(0)=u_0,\label{0}
\end{equation}
where 
\[ A=\left( \begin{array}{ccc}
\alpha & -\alpha & 0 \\
\alpha & 1 & 0 \\
0 & 0 & b \end{array} \right), \quad f=\left( \begin{array}{c}
0\\
0\\
-b(r+\alpha)
\end{array}\right)\]

\[
B(u,\tilde{u})=\left(\begin{array}{c}
0\\
(u_x\tilde{u_z}+u_z\tilde{u_x})/2\\
-(u_x\tilde{u_y}+u_y\tilde{u_x})/2\\
\end{array}\right).\]
We use the notation $\langle\cdot,\cdot\rangle$
and $|\cdot|$ for the standard Euclidean inner-product
and norm.
When describing our observations it will also be useful
to employ the projection matrices $P$ and $Q$ defined by 
\begin{equation}\label{31}
 P=\left( \begin{array}{ccc}
1 & 0 & 0 \\
0 & 0 & 0 \\
0 & 0 & 0 \end{array} \right) \quad Q=\left( \begin{array}{ccc}
0 & 0 & 0 \\
0 & 1 & 0 \\
0 & 0 & 1 \end{array} \right).
\end{equation}

We will use the following properties of $A$ and $B$:
\begin{properties} [\cite{hayden2011discrete}]
For all $u,\tilde{u} \in \mathbb{R}^3$
\label{p:2.1}
\mbox{}
\begin{enumerate}
 \renewcommand{\theenumi}{(\arabic{enumi})}
\item $\langle Au,u\rangle=\alpha u_{x}^{2}+u_{y}^{2}+bu_{z}^{2}> {|u|}^2$ provided that $\alpha,b>1.$ \label{itm:a}
\item $\langle B(u,u),u\rangle=0$.\label{itm:b}
\item $B(u,\tilde{u})=B(\tilde{u},u)$.\label{itm:c}
\item $|B(u,\tilde{u})|\leq 2^{-1}|u||\tilde{u}|$.\label{itm:d}
\item $|\langle B(u,\tilde{u}),\tilde{u}\rangle|\leq 2^{-1}|u||\tilde{u}||P\tilde{u}|$.\label{itm:e}
\end{enumerate}
\end{properties}

We will also use the following:

\begin{proposition} (\cite{hayden2011discrete}, Theorem 2.2)
Equation (\ref{0}) has a global attractor $\mathcal{A}$. Let $u$ be a trajectory with $u_0\in \mathcal{A}$. Then $|u(t)|^2\leq K$ for all $t\in \mathbb{R}$ where 
\begin{equation}\label{K}
K=\frac{b^2(r+\alpha)^2}{4(b-1)}.
\end{equation}
\end{proposition}

Figure \ref{fig:fig1} illustrates the
properties of the equation.
Sub-figure \ref{fig:subfig11} shows the global attractor
${\mathcal A}$. Sub-figures \ref{fig:subfig21}, \ref{fig:subfig31} and \ref{fig:subfig41} show the components $u_x$, $u_y$ and $u_z$, respectively, plotted against time.

\begin{figure}[h!]
\centering

\subfigure[Sample trajectory for a solution lying on global attractor.]{
\includegraphics[width=0.45\textwidth]{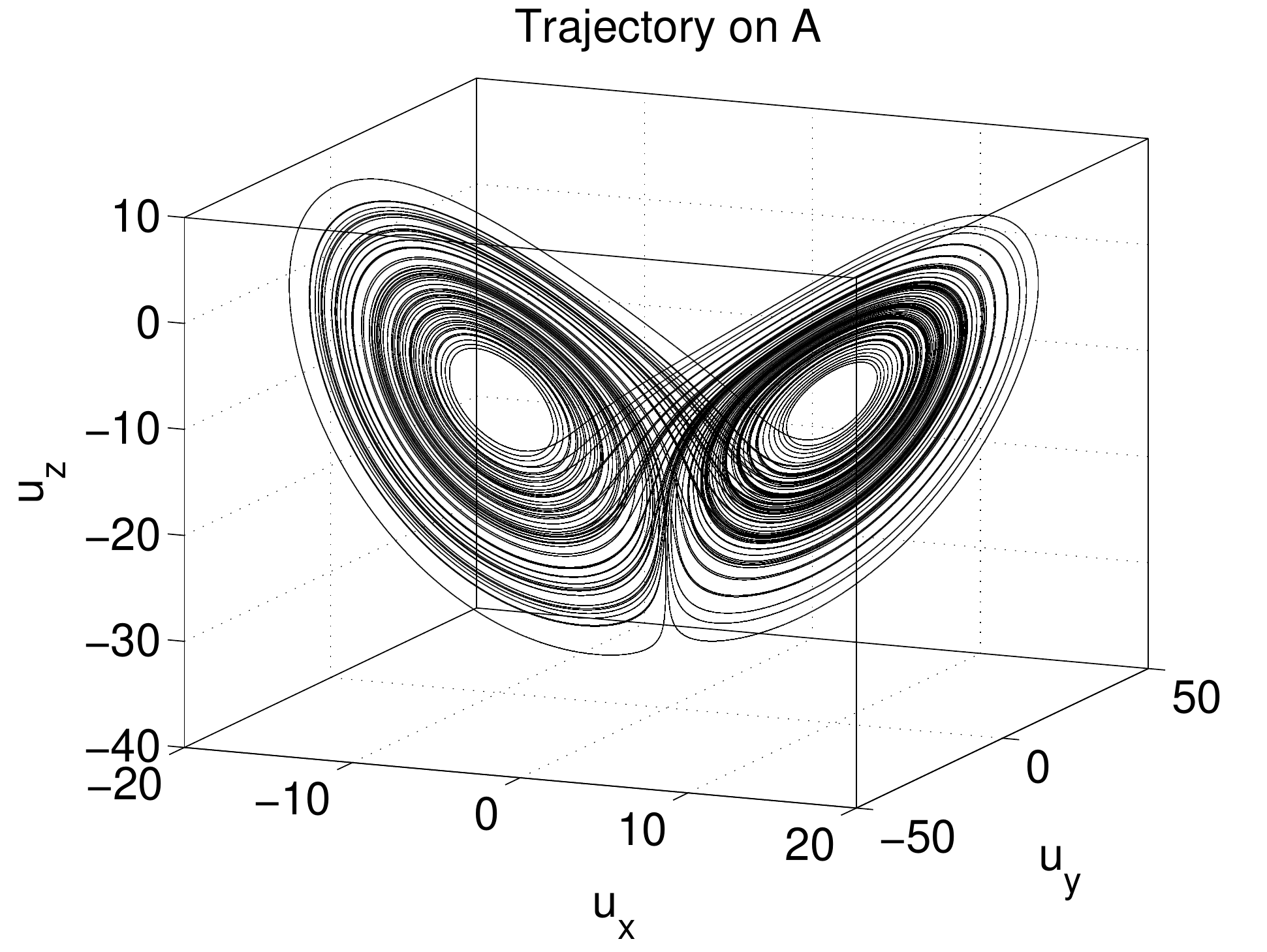} 
\label{fig:subfig11}
}
\subfigure[$u_x$ component]{
\includegraphics[width=0.45\textwidth]{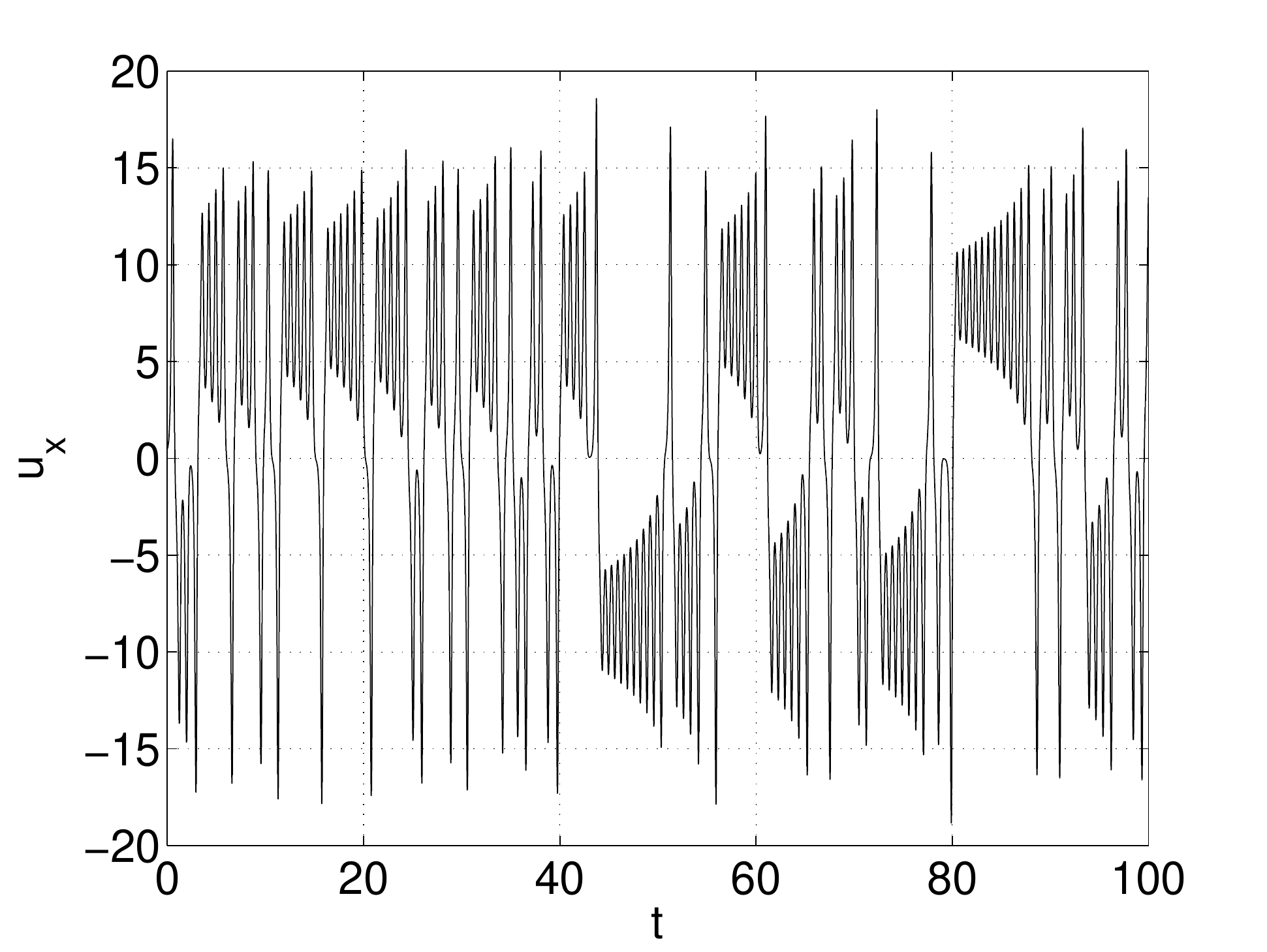} 
\label{fig:subfig21}
}
\subfigure[$u_y$ component.]{
\includegraphics[width=0.45\textwidth]{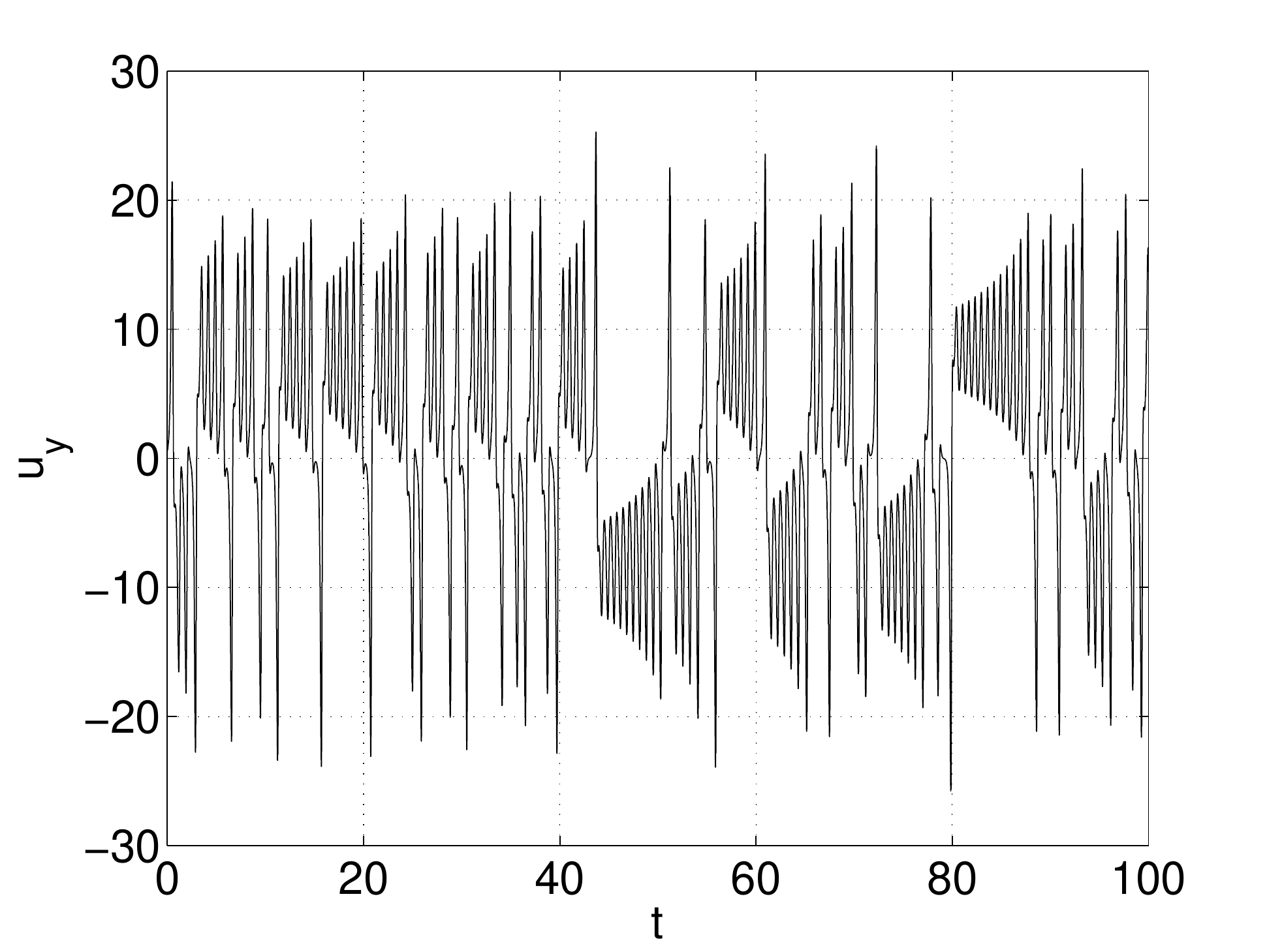} 
\label{fig:subfig31}
}
\subfigure[$u_z$ component.]{
\includegraphics[width=0.45\textwidth]{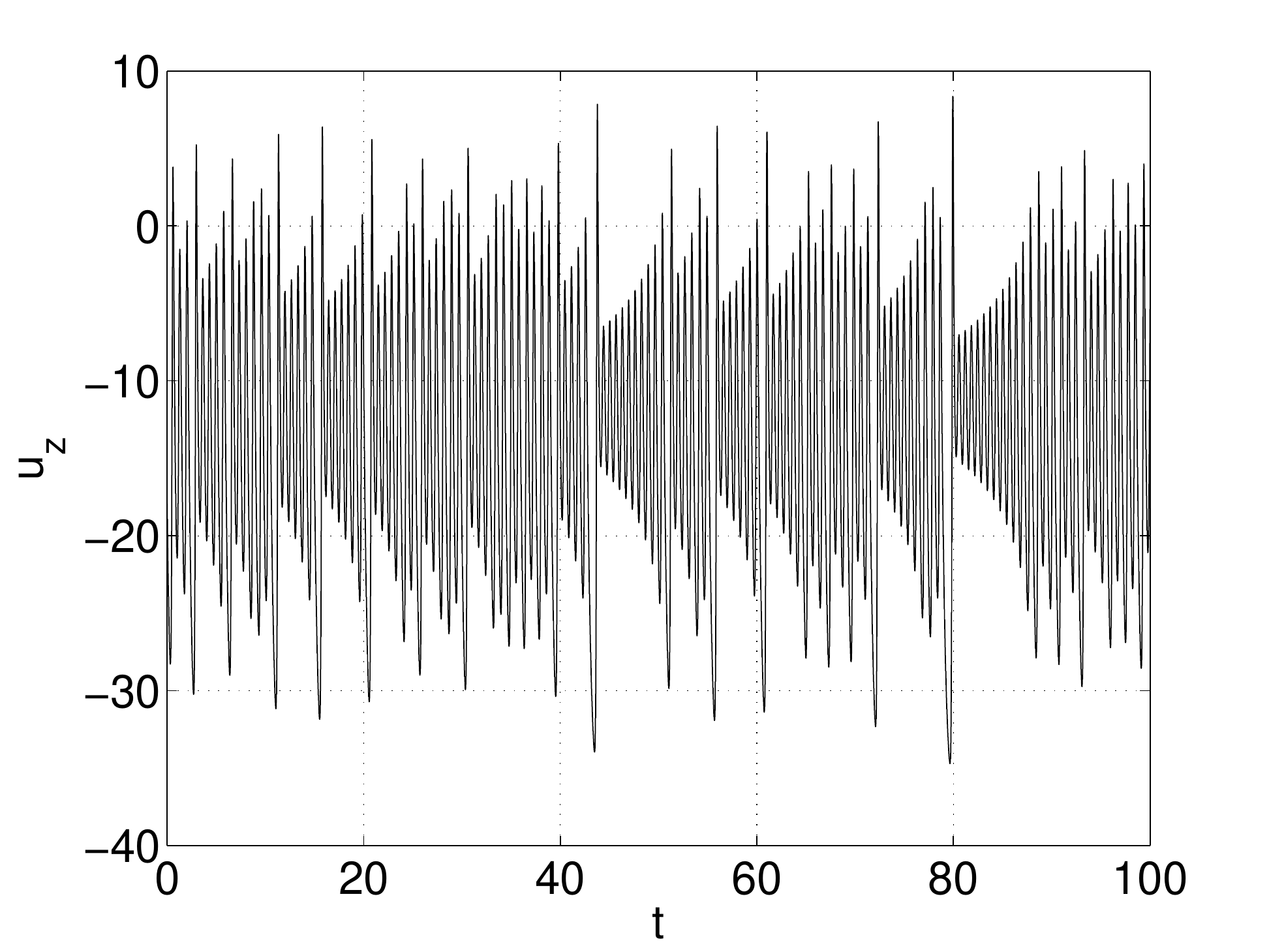} 
\label{fig:subfig41}
}

\caption{Lorenz attractor and individual components.}\label{fig:fig1}
\end{figure}

\subsection{3DVAR: Discrete Time Data}
\label{ssec:2.2}

In this section we describe the 3DVAR filtering scheme 
for the model (\ref{211}) in the case where 
the system is observed discretely at equally spaced time points.
The system state at time $t_k=kh$ is denoted by
$u_k=u(t_k)$ and the data upto that time
is $Y_k=\{{y_j}\}_{j=1}^k.$
Recall that our aim is to find the probability distribution of $u_k|Y_k$. Approximate Gaussian filters, of which 3DVAR is
a prototype, impose the following approximation: 
\begin{equation}\label{32}
\mathbb{P}(u_k|Y_k)= N(m_k,C_k).
\end{equation}
Given this assumption the filtering scheme can be written as an update rule 
\begin{equation}\label{33}
(m_k,C_k)\mapsto(m_{k+1},C_{k+1}).
\end{equation}
To determine this update we make a further Gaussian 
approximation, namely that $u_{k+1}$ 
given $Y_k$ follows a Gaussian distribution:
\begin{equation}\label{34}
\mathbb{P}(u_{k+1}|Y_k)= N(\hat{m}_{k+1},\hat{C}_{k+1}).
\end{equation}
Now we can break the update rule into two steps of {\em prediction} $(m_k,C_k)\mapsto(\hat{m}_{k+1},\hat{C}_{k+1})$ and {\em analysis} $(\hat{m}_{k+1},\hat{C}_{k+1})\rightarrow(m_{k+1},C_{k+1})$. For the prediction step we assume that  $\hat{m}_{k+1}=\Psi(m_k)$ whilst the choice of the covariance matrix $\hat{C}_{k+1}$ depends upon the choice of particular approximate Gaussian filter under consideration. For the analysis step, (\ref{34}) together with the fact that $y_{k+1}|u_{k+1}\sim N(Hu_{k+1},\Gamma)$
and application of Bayes' rule, implies that
\begin{equation}\label{35}
u_{k+1}|Y_{k+1}\sim N({m}_{k+1},{C}_{k+1})
\end{equation}
where \cite{harvey1991forecasting}
\begin{subequations}
\begin{eqnarray}
C_{k+1}&=&\hat{C}_{k+1}-\hat{C}_{k+1}H^*(\Gamma + H\hat{C}_{k+1}H^*)^{-1}H\hat{C}_{k+1}\label{36a}\\
m_{k+1}&=&\Psi(m_k)+\hat{C}_{k+1}H^*(\Gamma+H\hat{C}_{k+1}H^*)^{-1}\bigl(y_{k+1}-
H\Psi({m}_{k})\bigr).\label{36b}
\end{eqnarray}
\end{subequations}
As mentioned the choice of update rule $C_k\rightarrow \hat{C}_{k+1}$ defines the particular approximate Gaussian filtering scheme. For the 3DVAR scheme we impose $\hat{C}_{k+1}=C\;\; \forall k\in\mathbb{N}$ where $C$ is a positive definite 
$p \times p$ matrix. From equation (\ref{36b}) we then get
\begin{eqnarray}
m_{k+1}&=&\Psi(m_k)+CH^*{(\Gamma + HCH^*)}^{-1}\bigl(y_{k+1}-H\Psi(m_k)\bigr)\notag\\
&=&(I-GH)\Psi(m_k)+Gy_{k+1} \label{39}
\end{eqnarray}
where 
\begin{equation}
\label{kalman}
G:=CH^*{(\Gamma + HCH^*)}^{-1}
\end{equation}
is called Kalman gain matrix.
{The iteration} \eqref{39} is analyzed in section \ref{sec:dt}.

Another way of defining the 3DVAR filter is by means of
the following variational definition:
\begin{equation}
m_{k+1}= \underset{m}{\rm argmin}\;\left(\frac{1}{2}{\|C^{-\frac{1}{2}}\bigl(m-\Psi(m_k)\bigr)\|}^2+\frac{1}{2}{\|\Gamma^{-\frac{1}{2}}(y_{k+1}-Hm)\|}^2\right).
\end{equation} 
This coincides with the previous definition because the mean
of a Gaussian can be characterized as the minimizer of the
negative of the logarithm of the probability density function {and
because the analysis step corresponds to a Bayesian Gaussian update, 
given the assumptions underlying the filter; indeed the fact that
the negative logarithm is the sum of two squares follows from Bayes' theorem.}
From the variational formulation, it is clear that
the 3DVAR filter is a compromise between fitting the
model and the data. The model uncertainty is characterized
by a fixed covariance $C$, and the data uncertainty by
a fixed covariance $\Gamma$; the ratio of the size
of these two covariances will play an important role
in what follows.

\subsection{3DVAR: Continuous Time Data}
\label{ssec:2.3}

In this section we describe the limit of high frequency 
observations $h \to 0$ which, with appropriate scaling of the
noise covariance with respect to the observation time $h$,
leads to a stochastic differential equation (SDE) limit 
for the 3DVAR filter. We refer to this SDE as the continuous
time 3DVAR filter. We give a brief derivation, referring
to \cite{blomker2012accuracy} for further details and to \cite{reichbergman} for a related analysis of continuous time
limits in the context of the ensemble Kalman filter.

We assume the following scaling for the observation 
error covariance matrix: $\Gamma = \frac{1}{h}\Gamma_0$. Thus,
although the data arrives more and more frequently, as
we consider the limit $h \to 0$, it is also becoming
more uncertain; this trade-off leads to the SDE limit.
Define the sequence of variables $\{z_k\}_{k\in\mathbb{N}}$ by the relation $z_{k+1}=z_k+hy_{k+1}$ and $z_0=0$. Then
\begin{eqnarray}\label{38}
z_{k+1}=z_k+hHu_{k+1}+\sqrt{h\Gamma_0}\gamma_k, \quad z_0=0.
\end{eqnarray}
Here $\gamma_k\sim N(0,I)$. By rearranging and taking limit as $h\rightarrow 0$ we get
\begin{equation}\label{41}
\frac{\textrm{d}z}{\textrm{d}t}=Hu+\sqrt{\Gamma_0}\frac{\textrm{d}w}{\textrm{d}t},
\end{equation}
where $w$ is an ${\mathbb R}^m$ valued standard Brownian motion.
We think of $Z(t):=\{z(s)\}_{s \in [0,t]}$ as being
the data. For each fixed $t$ we have the jointly
varying random variable $(u(t),Z(t)) \in \bbR^p \times
C([0,t];\bbR^m).$ We are interested in the filtering problem
of determining the sequence of conditioned 
probability distributions
implied by the random variable $u(t)|Z(t)$ in $\bbR^p.$ 
The 3DVAR filter imposes Gaussian approximations 
of the form $N\bigl(m(t),C\bigr).$ We now derive
the evolution equation for $m(t).$

Recall the vector field $\mathcal{F}$ which drives equation \eqref{211}.
Using equation (\ref{38}) in (\ref{39}), together with the fact that $\Psi(u)=u+h\mathcal{F}(u)+{\mathcal O}(h^2)$, gives
\begin{equation}
m_{n+1}=m_n+h\mathcal{F}(m_n)+\mathcal{O}(h^2)+hCH^*(\Gamma_{0}+hHCH^*)^{-1}\left(\frac{z_{n+1}-z_n}{h}-Hm_n\right).
\end{equation}
Rearranging and taking limit $h\rightarrow 0$ gives
\begin{equation}\label{37}
\frac{\textrm{d}m}{\textrm{d}t}=\mathcal{F}(m)+CH^*\Gamma_{0}^{-1}\left(\frac{\textrm{d}z}{\textrm{d}t}-Hm\right).
\end{equation}
Equation (\ref{37}) defines the continuous time 3DVAR filtering scheme and is analyzed in section \ref{sec:ct}.
The data should be viewed as the continuous time stream
$Z(t)=\{z(s)\}_{s \in [0,t]}$ and equations \eqref{41} and \eqref{37} as
stochastic differential equations driven by $w$ and $z$
respectively.

\section{Analysis of Discrete Time 3DVAR}
\label{sec:dt}

In this section we analyse the discrete time 3DVAR algorithm when applied to a partially observed Lorenz '63 model; in particular we assume only that the $u_x$ component is observed. We start, in subsection \ref{ssec:3.1}, with some general discussion of 
error propagation properties of the filter. In subsection \ref{ssec:3.2} we study mean square behaviour of the filter for Gaussian noise. 
Recall the projection matrices $P$ and $Q$ given by (\ref{31}), we will use these in the following.
We will also use $\{v_k\}$ to denote the exact
solution sequence from the Lorenz equations
which underlies the data; this is to be contrasted
with $\{u_k\}$ which denotes the random variable
which, when conditioned on the data, is 
approximated by the 3DVAR filter.

\subsection{Preliminary Calculations}
\label{ssec:3.1}
Throughout we assume that {$H=(1,0,0)$, so that only $u_x$ is observed,
and we choose the model covariance $C=\eta^{-1}\epsilon^2 I$. 
We also assume that $\Gamma = {\epsilon}^2$. 
The Kalman gain matrix is then $G=\frac{1}{1+\eta}H^*$ and
the 3DVAR filter \eqref{39} may be written 
\begin{equation}
m_{k+1} = \left( \frac{\eta}{1+\eta}P+Q\right)\Psi(m_k)+\frac{1}{1+\eta}y_{k+1}H^*.
\label{39a}
\end{equation}
The scalar parameter is a design parameter whose choice
we will discuss through the analysis of the iteration \eqref{39a}.
Note that we are working with rather specific choices of
model and observational noise covariances $C$ and $\Gamma$;
we will comment on generalizations in the concluding section
\ref{sec:conc}.}

We define $v$ to be the true solution of the Lorenz equation 
(\ref{0}) which underlies the data, and we define $v_k=v(kh)$, 
the solution at observation times. 
Note that, since $\Gamma = {\epsilon}^2$, it is consistent to assume 
that the observation errors have the form  
\[\nu_k=\left( \begin{array}{c}
\epsilon \xi_k\\
0\\
0
\end{array}\right),\]
where $\xi_k$ are i.i.d. random variables on ${\mathbb R}$. 
We will consider the case $\xi_1\sim N(0,1)$ for simplicity
of exposition.  Note that we may write
\begin{eqnarray*}
y_{k+1}H^*&=&Pv_{k+1}+\nu_{k+1}\\
&=&P\Psi(v_k)+\nu_{k+1}.
\end{eqnarray*}
Thus
\begin{equation}
\label{eq:me}
m_{k+1}=\left(\frac{\eta}{1+\eta}P+Q\right)\Psi(m_k) + \frac{1}{1+\eta}\Bigl(P\Psi(v_k)+\nu_{k+1}\Bigr).
\end{equation}
Observe that
\begin{equation}
v_{k+1}=\Psi(v_k)
=\left(\frac{\eta}{1+\eta}P+Q\right)\Psi(v_k)+\frac{1}{1+\eta}P\Psi(v_k).
\label{eq:ve}
\end{equation}

We are interested in comparing $m_k$, the output of the
filter, with $v_k$ the true signal which underlies the data.
We define the error process $\delta(t)$ as follows:
\[
 \delta(t)=\left\{ \begin{array}{c}
m_k-v(t)\quad \textrm{if}\:\: t=t_k\\
\Psi(m_k,t-t_k)-v(t)\quad \textrm{if}\:\: t\in(t_{k},t_{k+1})\\
\end{array}\right.
\]
Observe that $\delta$ is discontinuous at times
$t_j$ which are multiples of $h$, since $m_{k+1} \ne
\Psi(m_k;h).$ In the following
we write $\delta(t_j^-)$
for $\lim_{t\rightarrow t_{j}^-}\delta(t)$ and we
define $\delta_j=\delta(t_j)$.
Thus $\delta_j \ne \delta(t_j^-).$ 
Subtracting \eqref{eq:ve} from \eqref{eq:me} we obtain
\begin{equation}
\delta(t_{k+1})=\left(\frac{\eta}{1+\eta}P+Q\right)\delta(t_{k+1}^-)+\frac{1}{1+\eta}\nu_k \label{89}.
\end{equation}
Now consider the time interval $(t_k,t_{k+1}).$ Since $\delta(t)$
is simply given by the difference of two solutions of the
Lorenz equations in this interval, we have
\begin{equation}
\frac{d\delta}{dt}+A\delta+B(v,\delta)+B(\delta ,v)+B(\delta,\delta)=0,\quad t\in(t_k,t_{k+1}). \label{1}
\end{equation}
Taking the Euclidean inner product of equation (\ref{1}) with $\delta$ gives
\begin{equation}
\frac{1}{2}\frac{d{|\delta|}^2}{dt}+\langle A\delta , \delta\rangle+\langle B(v,\delta), \delta\rangle+\langle B(\delta ,v), \delta\rangle +\langle B(\delta,\delta), \delta\rangle=0
\end{equation}
which, on simplifying and using Properties \ref{p:2.1}, gives
\begin{equation}
\frac{1}{2}\frac{d{|\delta|}^2}{dt}+\langle A\delta , \delta\rangle+ 2\langle B(v,\delta), \delta\rangle=0\label{2},
\end{equation}
and hence
\begin{equation}
\frac{1}{2}\frac{d{|\delta|}^2}{dt}+{|\delta|}^{2}+ 2\langle B(v,\delta), \delta\rangle\leq 0.\label{7}
\end{equation}

In order to use \eqref{89} we wish to estimate
the behaviour of $\delta(t_{k+1}^-)$ in terms of
$\delta_k$. The following is useful in this regard and
may be proved by using \eqref{7} together with
Properties \ref{p:2.1}(4). Note that $K$ is defined by equation (\ref{K})
and is necessarily greater than or equal to one, since $b, \alpha > 1.$ 

\begin{proposition}[\cite{hayden2011discrete}]\label{11}
Assume the true solution $v$ lies on the global attractor ${\mathcal A}$ so that 
${\rm sup}_{t \ge 0}|v(t)|^2\leq K$ with 
\[ K=\frac{b^2(r+\alpha)^2}{4(b-1)}.
\]
Then for $\beta = 2\left(K^{1/2}-1\right)$ it follows that ${|\delta (t)|}^{2}\leq {|\delta_k|}^{2}e^{\beta (t-t_k)}$ for $t\in [t_k , t_{k+1})$.
\end{proposition}

\subsection{Accuracy Theorem}
\label{ssec:3.2}

In this subsection we assume that $\xi_1\sim N(0,1)$ and
we study the behaviour of the filter in forward
time when the size of the observational noise, 
${\cal O}(\epsilon)$, is small.
The following result shows that, provided variance
inflation is employed ($\eta$ small enough), the 3DVAR filter
can recover from an ${\cal O}(1)$ initial error and enter
an ${\cal O}(\epsilon)$ neighbourhood 
of the true signal. The results are
proved in mean square. The reader will observe that
the bound on the error behaves poorly as the observation
time $h$ goes to zero, a result of the over-weighting of
observed data which is fluctuating wildly as $h \to 0.$
This effect is removed in section \ref{sec:ct} where the
observational noise is scaled appropriately, in terms
of $h \to 0$, to avoid this effect.

For this theorem we define a norm $\|\cdot\|$ by $\|u\|^2=|u|^2+|Pu|^2$,
where $|\cdot|$ is the Euclidean norm.

\begin{theorem}\label{91}
Let $v$ be a solution of the Lorenz equation (\ref{0}) with $v(0)\in \mathcal{A}$, the global attractor. Assume that $\xi_1\sim N(0,1)$ so that the observational noise is Gaussian. Then there exist $h_c>0,\;\lambda>0$ such that for all $\eta$ sufficiently small and all $h\in (0,h_c)$
\begin{equation}
\mathbb{E}{||\delta_{k+1}||}^2\le (1-\lambda h)\mathbb{E}{||\delta_k||}^2+2{\epsilon}^2.
\end{equation}
Consequently
\begin{equation}\label{58}
\underset{k\rightarrow\infty}{\limsup}\,\mathbb{E}{||\delta_k||}^2\le \frac{2{\epsilon}^2}{\lambda h}.
\end{equation}
\end{theorem}

\begin{proof}
Recall that we have $\mathbb{E}\nu_{k+1}=0$ and $\mathbb{E}{|\nu_{k+1}|}^2={\epsilon}^2$. On application of the projection
$P$ to the error 
equation (\ref{89}) for 3DVAR we obtain 
\begin{equation}
\mathbb{E}{|P\delta_{k+1}|^2} \leq {\left(\frac{\eta}{1+\eta}\right)}^2\mathbb{E}{|P\delta(t_{k+1}^-)|}^2+ {\left(\frac{1}{1+\eta}\right)}^2\epsilon^2.\label{10}\\
\end{equation}
Since $\mathbb{E}{|Q\delta_{k+1}|^2}=\mathbb{E}{|Q\delta(t_{k+1}^-)|}^2 \le \mathbb{E}{|\delta(t_{k+1}^-)|}^2$ 
we also obtain the bound
\begin{equation}
\mathbb{E}{|\delta_{k+1}|^2} \leq {\left(\frac{\eta}{1+\eta}\right)}^2\mathbb{E}{|P\delta(t_{k+1}^-)|}^2+
\mathbb{E}{|\delta(t_{k+1}^-)|}^2+ {\left(\frac{1}{1+\eta}\right)}^2\epsilon^2.\label{10a}\\
\end{equation}
Define $M_1$ and $M_2$ by 
\begin{equation}
M_1(\tau)=\frac{K\alpha}{\beta+\alpha}\left(\frac{e^{\beta\tau}-e^{-\tau}}{\beta+1}-\frac{e^{-\alpha\tau}-e^{-\tau}}{1-\alpha}\right)+e^{-\tau} 
+2{\left(\frac{\eta}{1+\eta}\right)}^2\left(\frac{\alpha}{\beta+\alpha}\right)(e^{\beta \tau}-e^{-\alpha \tau})
\end{equation}
and
\begin{equation}
M_2(\tau)= \frac{K}{1-\alpha}\left(e^{-\alpha\tau}-e^{-\tau}\right)+2{\left(\frac{\eta}{1+\eta}\right)}^2 e^{-\alpha \tau}.
\end{equation}

Adding \eqref{10} to \eqref{10a} and using Lemma \ref{l:sim}
shows that 
\begin{equation}\label{50}
\mathbb{E}{\|\delta_{k+1}\|}^2\le M_1(h)\mathbb{E}{|\delta_k|}^2 + M_2(h)\mathbb{E}{|P\delta_k|}^2+2{\left(\frac{1}{1+\eta}\right)}^2\epsilon^2,
\end{equation}
so that
\begin{equation}
\mathbb{E}{||\delta_{k+1}||}^2\le M(h)\mathbb{E}{||\delta_k||}^2+\frac{2\epsilon^2}{(1+\eta)^2},
\end{equation}
where
\begin{equation}
M(\tau)=\max \{M_1(\tau),M_2(\tau)\}.
\end{equation} 
Now we observe that
\begin{eqnarray*}
M_1(0)=1,\;
M_1'(0)=-1+2\alpha{\left(\frac{\eta}{1+\eta}\right)}^2\; \textrm{and}\; M_2(0)=2{\left(\frac{\eta}{1+\eta}\right)}^2.
\end{eqnarray*}

Thus there exists an $h_c>0$ and a $\lambda>0$ such that, for all $\eta$ sufficiently small
\begin{equation*}
M(\tau,\eta)\le 1-\lambda\tau,\quad \forall \tau\in (0,h_c].
\end{equation*}
Hence the theorem is proved.
\end{proof}

The following lemma is used in the preceding proof.

\begin{lemma}
\label{l:sim}
Under the conditions of Theorem \ref{91} for $t\in[t_k,t_{k+1})$ we have 
\begin{equation}
{|P\delta(t)|}^2\leq \frac{\alpha{|\delta_k|}^{2}}{\beta+\alpha}\left(e^{\beta (t-t_k)}-e^{-\alpha (t-t_k)}\right)+{|P\delta_k|}^2e^{-\alpha (t-t_k)}\label{5}
\end{equation}
and 
\begin{equation}
\begin{split}
|\delta(t)|^2 \leq& \frac{K\alpha{|\delta_k|}^{2}}{\beta+\alpha}\left(\frac{e^{\beta(t-t_k)}-e^{-(t-t_k)}}{\beta+1}-\frac{e^{-\alpha(t-t_k)}-e^{-(t-t_k)}}{1-\alpha}\right) \\
 & \quad + \frac{K{|P\delta_k|}^2}{1-\alpha}\left(e^{-\alpha(t-t_k)}-e^{-(t-t_k)}\right) +|\delta_k|^2e^{-(t-t_k)}.\label{9}
\end{split}
\end{equation}
\end{lemma}

\begin{proof}
Taking inner product of (\ref{1}) with $P\delta$, instead
of with $\delta$ as previously, we get
\begin{equation}
\frac{1}{2}\frac{d{|P\delta|}^2}{dt}+\langle A\delta , P\delta\rangle=0.\label{3}
\end{equation}
Let $\delta={(\delta_x, \delta_y, \delta_z)}^T$. Notice that ${|P\delta|}^2={|\delta_x|}^2$ and
$\langle A\delta , P\delta\rangle = \alpha\delta_{x}^{2}-\alpha\delta_x\delta_y.$
Therefore equation (\ref{3}) becomes 
\begin{eqnarray*}
\frac{1}{2}\frac{d{|P\delta|}^2}{dt}+\alpha\delta_{x}^{2}&=&\alpha\delta_x\delta_y\\
&\leq& \frac{\alpha}{2}\delta_{x}^{2}+\frac{\alpha}{2}\delta_{y}^{2}\\
&\leq& \frac{\alpha}{2}\delta_{x}^{2}+\frac{\alpha}{2}{|\delta|}^2.
\end{eqnarray*}
By rearranging and applying Proposition \ref{11} we get
\begin{equation}
\frac{d{|P\delta|}^2}{dt}+\alpha{|P\delta|}^2\leq \alpha{|\delta(t_k)|}^{2}e^{\beta (t-t_k)}.
\end{equation}
Multiplying by integrating factor $e^{\alpha(t-t_k)}$ and integrating from $t_k$ to $t$ gives equation \eqref{5}.

Analysing the non-linear term in equation (\ref{7}) with 
Property \ref{p:2.1}(5) gives
\begin{eqnarray}
|2\langle B(v,\delta), \delta\rangle|&\leq& |v||P\delta||\delta|\\
&\leq& K^\frac{1}{2}|P\delta ||\delta|\\
&\leq& \frac{1}{2}K{|P\delta |}^2 + \frac{1}{2}{|\delta |}^2. \label{6}
\end{eqnarray} 
Substituting (\ref{5}) and (\ref{6}) in (\ref{7}) gives
\begin{equation}
\frac{d{|\delta|}^2}{dt}+{|\delta|}^{2}\leq \frac{K\alpha{|\delta_k|}^{2}}{\beta+\alpha}\left(e^{\beta (t-t_k)}-e^{-\alpha (t-t_k)}\right)+K{|P\delta_k|}^2e^{-\alpha (t-t_k)}.\label{8}
\end{equation}
Multiplying by the integrating factor $e^{(t-t_k)}$ and integrating from $t_k$ to $t$ gives
\begin{equation}
|\delta(t)|^2e^{(t-t_k)}-|\delta_k|^2\leq \frac{K\alpha{|\delta_k|}^{2}}{\beta+\alpha}\left(\frac{e^{(\beta+1)(t-t_k)}-1}{\beta+1}-\frac{e^{(1-\alpha)(t-t_k)}-1}{1-\alpha}\right)+\frac{K{|P\delta_k|}^2}{1-\alpha}\left(e^{(1-\alpha)(t-t_k)}-1\right).
\end{equation}
Rearranging the above equation gives \eqref{9}.
\end{proof}

\section{Analysis of Continuous Time 3DVAR}
\label{sec:ct}

In this section we analyse application of the 3DVAR continuous filtering algorithm for the  Lorenz equation (\ref{0}). 
We will use $\{v(t)\}_{t \in [0,\infty)}$ to denote the exact
solution sequence from the Lorenz equations
which underlies the data; this is to be contrasted
with $\{u(t)\}_{t \in [0,\infty)}$ 
which denotes the random variable
which, when conditioned on the data, is 
approximated by the 3DVAR filter.

We study the continuous time 3DVAR filter, again in the
case where $H=(1,0,0)$, $\Gamma_0={\epsilon}^2$ 
and {$C = \eta^{-1}{\epsilon}^2 I$}.
To analyse the filter it is useful to have the truth $v$
which gives rise to the data appearing in the filter itself.
Thus \eqref{41} gives 
\begin{equation}\label{41a}
\frac{\textrm{d}z}{\textrm{d}t}=Hv+\sqrt{\Gamma_0}\frac{\textrm{d}w}{\textrm{d}t}.
\end{equation}
We then eliminate $z$ in equation (\ref{37}) by
using \eqref{41a} to obtain
\begin{equation}\label{51}
\frac{\textrm{d}m}{\textrm{d}t}=\mathcal{F}(m)+CH^*\Gamma_{0}^{-1}H(v-m)+CH^*\Gamma_{0}^{-\frac{1}{2}}\frac{\textrm{d}w}{\textrm{d}t}.
\end{equation}
In the specific case of the Lorenz equation we get
\begin{equation}\label{55}
\frac{\textrm{d}m}{\textrm{d}t}=-Am-B(m,m)+f+CH^*\Gamma_{0}^{-1}H(v-m)+CH^*\Gamma_{0}^{-\frac{1}{2}}\frac{\textrm{d}w}{\textrm{d}t}.
\end{equation}
From equation (\ref{51}) with the choices of $C$, $H$
and $\Gamma_0$ detailed above we get
\begin{equation}\label{52}
\frac{\textrm{d}m}{\textrm{d}t}=-Am-B(m,m)+f+\frac{1}{\eta}P(v-m)+\frac{\epsilon}{\eta}P\frac{\textrm{d}w}{\textrm{d}t}
\end{equation}
where we have extended $w$ from a scalar Brownian motion to an $\mathbb{R}^3$-valued Brownian motion for notational convenience. 
This {SDE} has a unique global strong solution $m \in C([0,\infty);\bbR^3)$.
Indeed similar techniques used to prove the following result may be
used to establish this global existence result, by applying the
It\^o formula to $|m|^2$ and using the global existence theory
in \cite{Mao97}; we omit the details. Recall $K$ given by \eqref{K}.

\begin{theorem}
\label{t:55}
Let $m$ solve equation(\ref{52}) and let $v$ solve 
equation (\ref{0}) with initial data $v(0) \in {\mathcal A}$,
the global attractor, so that $\sup_{t \ge 0}|v(t)|^2
\le K.$ Then for $\eta K<4$ we obtain
\begin{equation}\label{45}
\mathbb{E}{|m(t)-v(t)|}^2\le e^{-\lambda t}{|m(0)-v(0)|}^2+\frac{\epsilon^2}{\eta^2\lambda}(1-e^{-\lambda t}),
\end{equation}
where $\lambda$ is defined by 
\begin{equation}\label{56}
\lambda=2\left(1-\frac{\eta K}{4}\right).
\end{equation}
Thus
$${\rm limsup}_{t \to \infty} \mathbb{E}{|m(t)-v(t)|}^2\le
\frac{\epsilon^2}{\lambda \eta^2}.$$
\end{theorem}

\begin{proof}
The true solution follows the model
\begin{equation}\label{53}
\frac{\textrm{d}v}{\textrm{d}t}=-Av-B(v,v)+f+\frac{1}{\eta}P(v-v),
\end{equation}
where we include the last term, which is identically zero, for
clear comparison with the filter equation (\ref{52}).  
Define $\delta = m-v$ and subtract equation (\ref{53}) 
from equation(\ref{52}) to obtain
\begin{eqnarray}
\frac{\textrm{d}\delta}{\textrm{d}t}&=&-Am-B(m,m)+Av+B(v,v)-\eta^{-1}P\delta+\epsilon\eta^{-1}P\frac{\textrm{d}w}{\textrm{d}t}\\
&=&-A\delta-2B(v,\delta)-B(\delta,\delta)-\eta^{-1}P\delta
+\epsilon\eta^{-1}P\frac{\textrm{d}w}{\textrm{d}t}.
\end{eqnarray}
Using It\^o's formula gives
\begin{equation}
\frac{1}{2}\textrm{d}{|\delta|}^2
+\langle A\delta+2B(v,\delta)+B(\delta,\delta)+\frac{1}{\eta}P\delta,\delta\rangle {\textrm{d}t}
\le \langle \epsilon\eta^{-1}P\textrm{d}w,\delta\rangle + \frac{1}{2}\tr\left(\epsilon^2\eta^{-2}P\right){\textrm{d}t} .
\end{equation}
Using Lemma \ref {57} and the definition of $\lambda$ gives
\begin{equation}
\frac{1}{2}\textrm{d}{|\delta|}^2+\frac{\lambda}{2}{|\delta|}^2{\textrm{d}t}\le \langle \epsilon\eta^{-1}P\textrm{d}w,\delta\rangle + \frac{1}{2}\tr\left(\epsilon^2\eta^{-2}P\right){\textrm{d}t} .
\end{equation}
Rearranging and taking expectations gives
\begin{equation}
\frac{\textrm{d}\mathbb{E}{|\delta|}^2}{\textrm{d}t}\le -\lambda\mathbb{E}{|\delta|}^2+\frac{\epsilon^2}{\eta^{2}}.
\end{equation}
Use of the Gronwall inequality gives the desired result.
\end{proof}

The following lemma is used in the preceding proof.

\begin{lemma}\label{57}
Let $v\in\mathcal{A}.$ Then
\begin{equation}
\langle A\delta+2B(v,\delta)+B(\delta,\delta)+\frac{1}{\eta}P\delta,\delta\rangle\ge \left(1-\frac{\eta K}{4}\right){|\delta|}^2.
\end{equation}
\end{lemma}
\begin{proof}

On expanding the inner product   
\begin{equation}
\langle A\delta+2B(v,\delta)+B(\delta,\delta)+\frac{1}{\eta}P\delta,\delta\rangle=\langle A\delta,\delta\rangle+2\langle B(v,\delta),\delta\rangle+\langle B(\delta,\delta),\delta\rangle+\langle \eta^{-1}P\delta,\delta\rangle.
\end{equation}
We now use the Properties \ref{p:2.1}(1),(5) and the fact that true solution lies on the global attractor so that $|v|\le K$.
As a consequence we obtain
\begin{equation}
\langle A\delta+2B(v,\delta)+B(\delta,\delta)+\frac{1}{\eta}P\delta,\delta \rangle \ge {|\delta|}^2-K^{\frac{1}{2}}|\delta||P\delta|+\frac{1}{\eta}{|P\delta|}^2.
\end{equation}
Using Young's inequality with parameter $\theta$
\begin{equation}
\langle A\delta+2B(v,\delta)+B(\delta,\delta)+\frac{1}{\eta}P\delta,\delta \rangle\ge {|\delta|}^2-\frac{1}{2\theta}K{|P\delta|}^2-\frac{\theta}{2}{|\delta|}^2+\frac{1}{\eta}{|P\delta|}^2.
\end{equation}
Taking $\theta=\frac{\eta K}{2}$ yields the desired result
\begin{equation}
\langle A\delta+2B(v,\delta)+B(\delta,\delta)+\frac{1}{\eta}P\delta,\delta\rangle\ge\left(1-\frac{\eta K}{4}\right){|\delta|}^2.
\end{equation}

\end{proof}
\section{Numerical Results}
\label{sec:num}

In this section we present numerical results illustrating 
Theorems \ref{91} and \ref{t:55}
established in the two preceding sections. 
All experiments are conducted with
the parameters $(\alpha,b,r)=(10,\frac{8}{3},28)$. 
Both the theorems are mean square results.
{However, some of our numerics are based on a single 
long-time realization of the filters in question, with time-averaging used
in place of ensemble averaging when mean square results
are displayed; we highlight when this is done.}

\subsection{Discrete case}
Under the assumptions of Theorem \ref{91} we expect the 
mean square error in {$\delta=|v-m|$} to decrease exponentially 
until it is of the size of the observational noise
squared. Hence we expect the estimate $m$ to converge to a 
neighbourhood of the true solution $v$, where the size 
of the neighbourhood scales as the size of the noise 
which pollutes in observation, in mean square. {The following
experiment indicates that similar behaviour is in fact
observed pathwise (Figure \ref{fig:fig2}), 
as well as in mean square over an ensemble (Figure \ref{fig:fig3}). 
We set up the numerical experiments by computing the true 
solution $v$ of the Lorenz equations using the
explicit Euler method, and then adding Gaussian random noise 
to the observed $x$-component to create the data.
Throughout we fix the parameter $\eta=0.1$. In
Figure \ref{fig:fig2}  the observational noise $\epsilon$
is fixed and in Figure \ref{fig:fig3} we vary it over a range
of scales.}

\begin {figure}
\begin{center}
\includegraphics[width=0.95\textwidth]{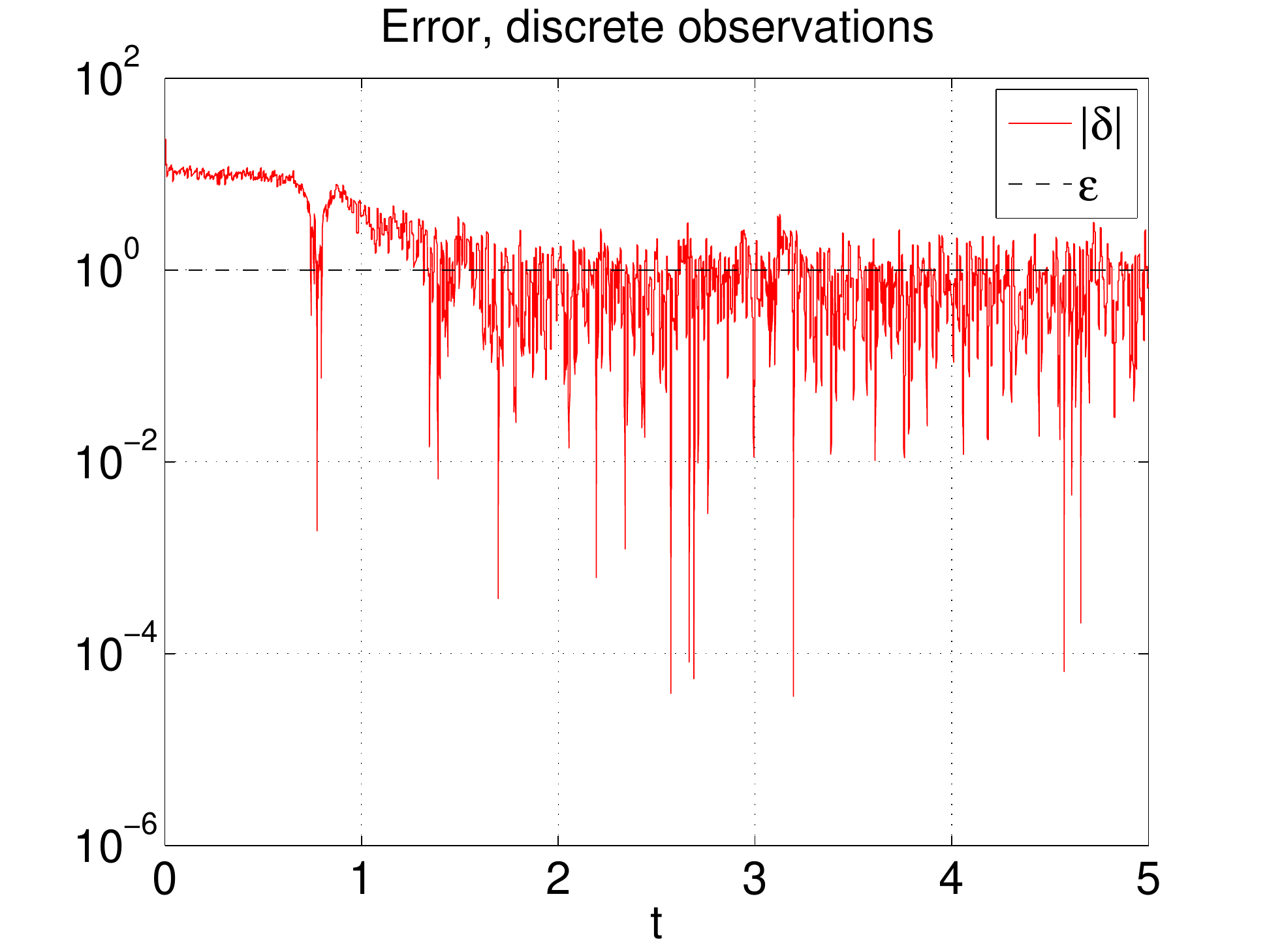} 
\end{center}
\caption{Decay of initial error from $\mathcal{O}(1)$ to
  $\mathcal{O}(\epsilon)$ for discrete observations, $\epsilon=1$, $\eta=0.1$}\label{fig:fig2}
\end {figure}

\begin {figure}
\begin{center}
\includegraphics[width=0.95\textwidth]{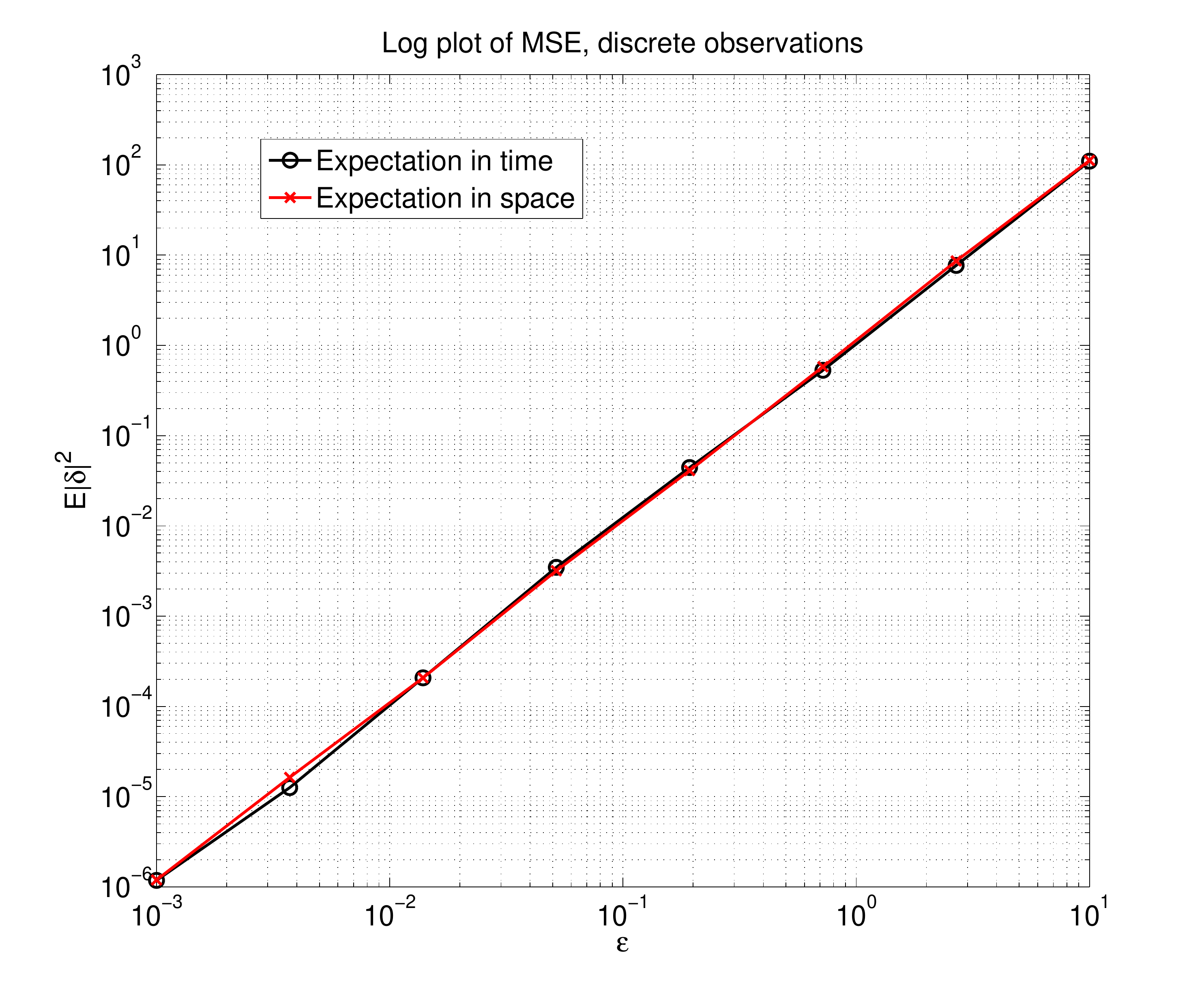} 
\end{center}
\caption{Log-linear dependence of asymptotic
  $\mathbb{E}{|\delta|^2}$ on $\epsilon$ for discrete
  observations, $\eta=0.1$.}\label{fig:fig3}
\end {figure}
Figure \ref{fig:fig2} {concerns the behaviour of a single
realization of the filter.} Note that the initial
error $|v(0)-m(0)|$ is 
around $\mathbb{E}|v| \approx 10$ and it  
decays exponentially with time, converging to
$\mathcal{O}(\epsilon)$;
 for this particular case we chose $\epsilon = 1.$ 
A consequence of the second part of Theorem \ref{91} is that the logarithm of the asymptotic mean squared error $\log\mathbb{E}{|\delta|^2}$ varies linearly with the logarithm of the standard deviation of noise in the observations $(\epsilon)$ {and this is illustrated
in Figure \ref{fig:fig3}.
To compute the asymptotic mean square error we take two approaches. 
In the first, for each $\epsilon$, we time-average the error incurred 
within a single long trajectory of the filter.  
In the second approach, we consider
spatial average over an ensemble of observational noises ${\nu}$, 
at a single time after the error has reached equilibrium.
In Figure \ref{fig:fig3}
we observe the log-linear decrease 
in the asymptotic error as the size of the noise decreases;
furthermore, the slope of the graph is approximately
$2$ as predicted by (\ref{58}). Both temporal and spatial
averaging deliver approximately the same slope.}

\subsection{Continuous case}
In the case of continuous observations we again compute
a true trajectory of the Lorenz equation using the
explicit Euler scheme. We then simulate the SDE  (\ref{55}) 
using the Euler-Maruyama method.\footnote{Note that this
is equivalent to creating the data $z$ from (\ref{41a})
and solving (\ref{37}) and, since we have access to the
truth, is computationally expedient.} {Similarly to the discrete case,
we consider both pathwise and ensemble illustrations 
of the mean square results in Theorem \ref{t:55}. 
Figures \ref{fig:fig4} and \ref{fig:fig6} concern
a single pathwise solution of (\ref{55}). Recall
from Theorem \ref{t:55} that the critical value of
$\eta$, beneath which the mean square theory holds,
is $\eta_{c}=4/K.$ In Figure \ref{fig:fig4} we have
$\eta=\frac12 \eta_c$ whilst in Figure \ref{fig:fig6} we
have $\eta=10 \eta_c;$ in both cases the pathwise error
spends most of its time at ${\cal O}(\epsilon)$, after
the initial transient is removed, suggesting that
the critical value of $\eta$ derived in Theorem \ref{t:55}
is not sharp. 
In Figure \ref{fig:fig5} we
vary the size of observational error $\epsilon$ and
take $\eta=\frac18 \eta_c.$
The initial error is expected to decay exponentially towards something
of order ${\cal \epsilon}$, and this is what is observed
in both the case where averaging is performed in time and in space. 
Indeed we observe the log-linear decrease 
in the asymptotic error as the 
size of the noise decreases, and the slope of the graph is 
approximately $2$, as predicted by equation (\ref{45}).}

\begin {figure}
\begin{center}
\includegraphics[width=0.95\textwidth]{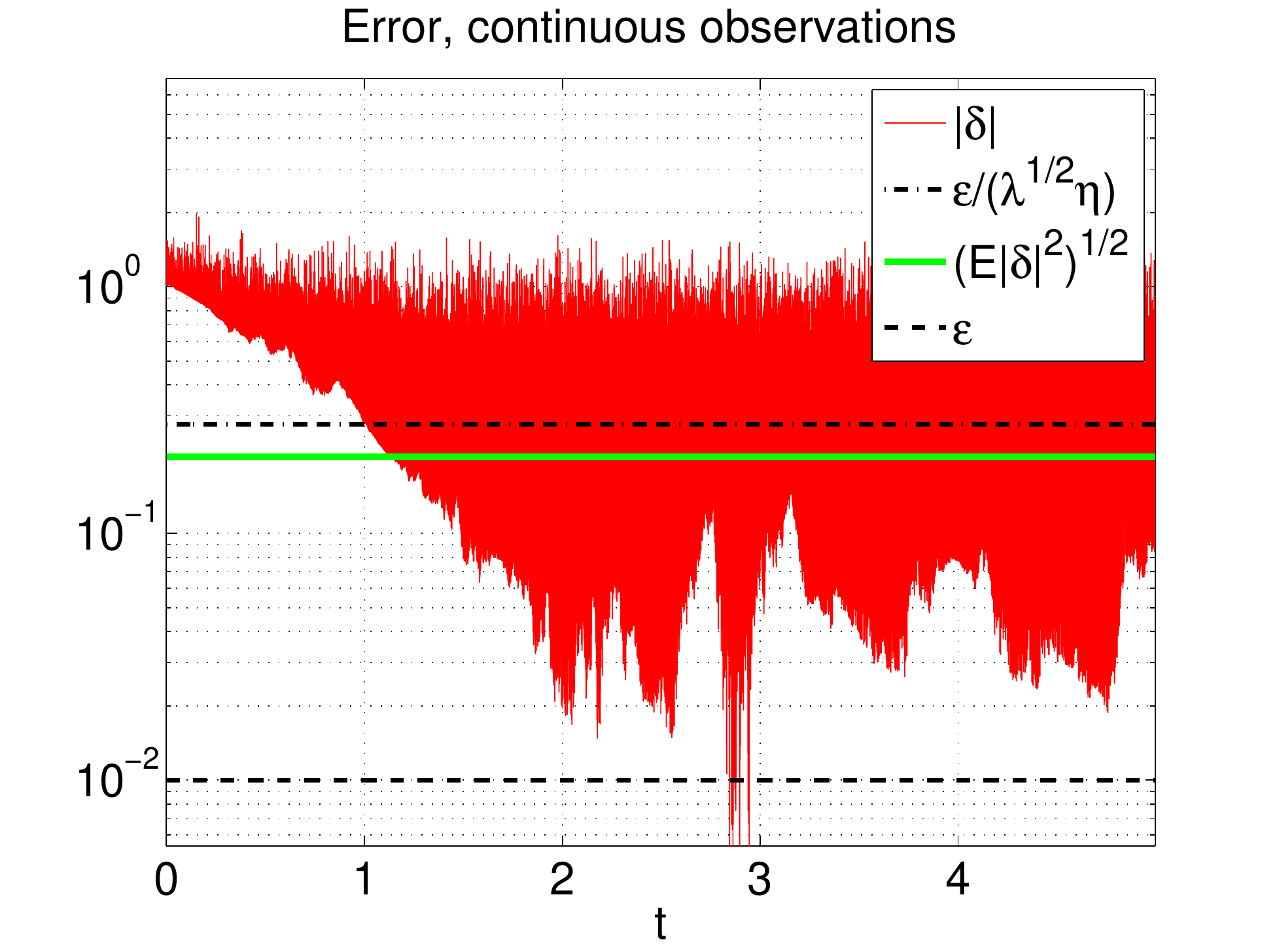} 
\end{center}
\caption{Decay of initial error from $\mathcal{O}(1)$ to
  $\mathcal{O}(\epsilon)$ for continuous observations, $\epsilon=0.01$.
Results are shown for $\eta=2/K<\eta_{c}$. }\label{fig:fig4}
\end {figure}

\begin {figure}
\begin{center}
\includegraphics[width=0.95\textwidth]{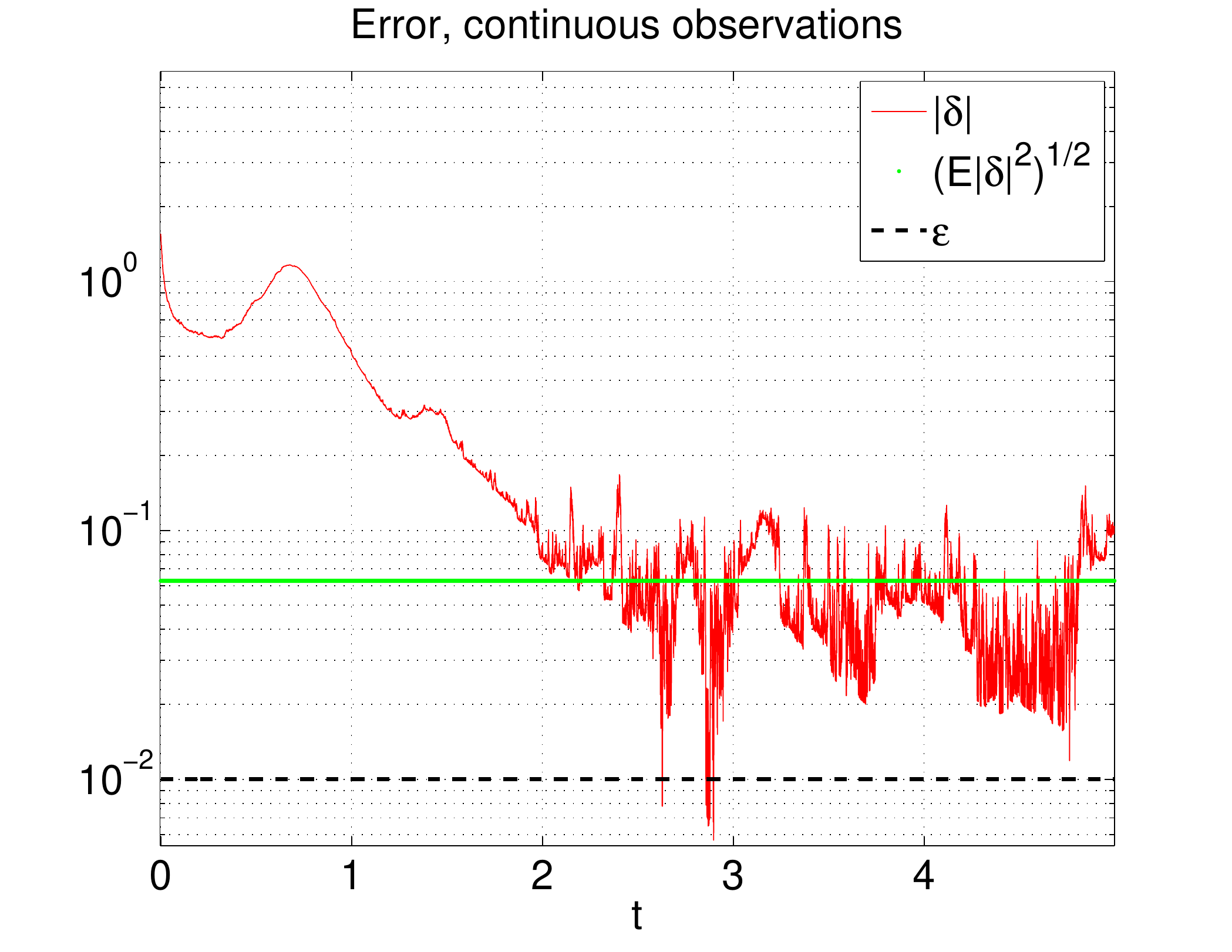} 
\end{center}
\caption{Decay of initial error from $\mathcal{O}(1)$ to
  $\mathcal{O}(\epsilon)$ for continuous observations, $\epsilon=0.01$.
Results are shown for $\eta=40/K=10\eta_{c}$. }\label{fig:fig6}
\end {figure}

\begin {figure}
\begin{center}
\includegraphics[width=0.95\textwidth]{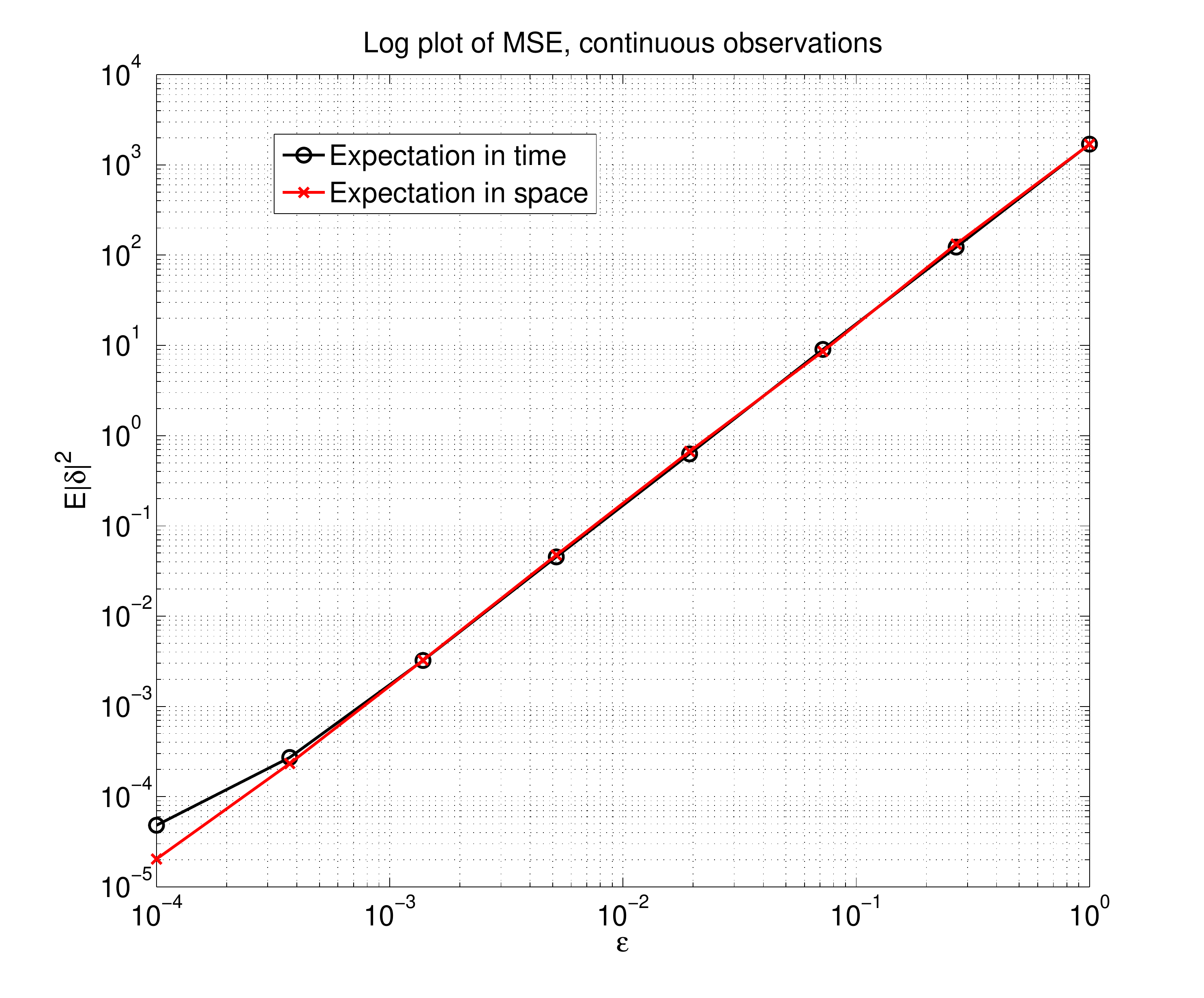} 
\end{center}
\caption{Log-linear dependence of asymptotic $\log\mathbb{E}{|\delta|^2}$
  on $\log \epsilon$ for continuous observations and $\eta=1/(2K)$.}
\label{fig:fig5}
\end {figure}

\section{Conclusions}
\label{sec:conc}

The study of approximate Gaussian filters for
the incorporation of data into high dimensional
dynamical systems provides a rich field for
applied mathematicians. Potentially such analysis
can shed light on algorithms currently in use,
whilst also suggesting methods for the improvement
of those algorithms. However, rigorous analysis
of these filters is in its infancy. The current
work demonstrates the properties of the 3DVAR
algorithm when applied to the partially observed
Lorenz '63 model; it is analogous to the
more involved theory developed for the 3DVAR
filter applied to the partially observed Navier-Stokes
equations in \cite{brett2012accuracy,blomker2012accuracy}. 
Work of this type can be built upon in
four primary directions: firstly to consider other
model dynamical systems of interest to practitioners,
such as the Lorenz '96 model \cite{lorenz1996predictability};
secondly to consider other observation models, such
as pointwise velocity field measurements or Lagrangian
data for the
Navier-Stokes equations, building on the theory of
determining modes \cite{jones1993upper};
{thirdly to consider the precise relationships required
between the model covariance $C$ and observation operator
$H$ to ensure accuracy of the filter}; 
and finally to consider more sophisticated filters
such as the extended \cite{jazwinski1970stochastic}
and ensemble \cite{evensen2003ensemble,evensen2009data}
Kalman filters.

{We are actively engaged in studying other models, such as
Lorenz '96, by similar techniques to those employed here;
our work on Lorenz '63 and Navier-Stokes models builds
heavily on the synchronization results of
Titi and coworkers and we believe that generalization of
synchronization properties is a key first step
in the study of other models. Regarding the second direction, Lagrangian
data introduces an additional auxiliary system for the observed
variables through which the system of interest is observed, 
necessitating careful design of correlations in the design parameters $C$, 
meaning that the analysis will be considerably more complicated
than for Eulerian data.
This links to the third direction: in general the relationship
between the model covariance and observation operator required to
obtain filter accuracy may be quite complicated and is an important
avenue for study in this field; even for the particular Lorenz '63 model
studied herein, with observation of only the $x$ component of
the system, this complexity is manifest if the covariance is not
diagonal. 
Relating to the fourth and final direction, it is worth noting that 
3DVAR is outdated operationally and 
empirical studies of filter accuracy have recently been focused 
on the more sophisticated methods such as 
ensemble Kalman filter and 4DVAR \cite{kalnay20084, lawstuart}.  
These empirical studies indicate that the more sophisticated methods 
outperform 3DVAR, as expected, and therefore suggest the importance
of rigorous analysis of those methods.}

\section*{Acknowledgement}
KJHL is supported by ESA and ERC.
AbS is supported by the EPSRC-MASDOC graduate
training scheme.
AMS is supported by ERC, EPSRC, ESA and ONR.
The authors are grateful to Daniel Sanz, Mike Scott
and Kostas Zygalakis for helpful comments 
on earlier versions of the manuscript.

\bibliographystyle{plain}
\bibliography{fsbib}

\end{document}